\RequirePackage{silence}
\documentclass[11pt,english]{article}

\usepackage[margin= 2.2 cm, bottom=19.5mm,footskip=12mm,top=20mm]{geometry}

\usepackage{amsthm}
\usepackage{amsmath}
\usepackage{amssymb}
\usepackage{setspace}
\usepackage{mathtools}
\usepackage{graphicx}
\usepackage[hidelinks]{hyperref}
\usepackage{cleveref}
\usepackage{graphicx}
\usepackage{enumerate}
\usepackage[title]{appendix}

\usepackage{floatrow}
\usepackage[T1]{fontenc}
\usepackage{tikz}

\theoremstyle{plain}

\newtheorem{thm}{Theorem}
\Crefname{thm}{Theorem}{Theorems}
\numberwithin{thm}{section}

\newtheorem{lem}[thm]{Lemma}
\Crefname{lem}{Lemma}{Lemmas}

\newtheorem{claim}{Claim}
\crefname{claim}{Claim}{Claims}
\Crefname{claim}{Claim}{Claims}

\newtheorem{prop}[thm]{Proposition}
\Crefname{prop}{Proposition}{Propositions}

\newtheorem{cor}[thm]{Corollary}
\crefname{cor}{Corollary}{Corollaries}

\newtheorem{conj}[thm]{Conjecture}
\crefname{conj}{Conjecture}{Conjectures}

\newtheorem{ex}[thm]{Example}
\Crefname{ex}{Example}{Examples}

\theoremstyle{definition}

\newtheorem{defn}[thm]{Definition}
\Crefname{defn}{Definition}{Definitions}

\newtheorem*{defn*}{Definition}

\newtheorem{nremark}[thm]{Remark}
\Crefname{nremark}{Remark}{Remarks}

\renewenvironment{proof}[1][]{\begin{trivlist}
\item[\hspace{\labelsep}{\bf\noindent Proof#1.\/}] }{\qed\end{trivlist}}

\newcommand{\eps}{\varepsilon}

\newcommand{\pr}{\mathbb{P}}
\newcommand{\M}{\mathcal{M}}

\expandafter\def\expandafter\normalsize\expandafter{%
    \normalsize
    \setlength\abovedisplayskip{4pt}
    \setlength\belowdisplayskip{4pt}
    \setlength\abovedisplayshortskip{4pt}
    \setlength\belowdisplayshortskip{4pt}
}

\usepackage[square,sort,comma,numbers]{natbib}
\newcommand{\EE}{\mathbb{E}}

\title{\vspace{-0.8cm} Large induced matchings in random graphs}

\author{Oliver Cooley\footnotemark[1]
\and
Nemanja Dragani\'c\footnotemark[2] 
\and
Mihyun Kang\footnotemark[1]
\and
Benny Sudakov\footnotemark[2] \footnotemark[3]
}

 \date{}

\begin{document}
\maketitle

\renewcommand{\thefootnote}{\fnsymbol{footnote}}

 \footnotetext[1]{Institute of Discrete Mathematics, Graz University of Technology, Steyrergasse 30, 8010 Graz, Austria. Emails: \href{mailto:cooley@math.tugraz.at} {\nolinkurl{cooley@math.tugraz.at}},
 \href{mailto:kang@math.tugraz.at} {\nolinkurl{kang@math.tugraz.at}}.
Supported by Austrian Science Fund (FWF): I3747.}

\footnotetext[2]{Department of Mathematics, ETH, Z\"urich, Switzerland. Emails:
\href{mailto:nemanja.draganic@math.ethz.ch} {\nolinkurl{nemanja.draganic@math.ethz.ch}},
\href{mailto:benjamin.sudakov@math.ethz.ch} {\nolinkurl{benjamin.sudakov@math.ethz.ch}}.}

\footnotetext[3]{Research supported in part by SNSF grant 200021\_196965.}

\renewcommand{\thefootnote}{\arabic{footnote}}

\begin{abstract}
    Given a large graph $H$, does the binomial random graph $G(n,p)$ contain a copy of $H$ as an induced subgraph with high probability? This classic question has been studied extensively for various graphs $H$,
    going back to the study of the independence number of $G(n,p)$ by
    Erd\H{o}s and Bollob\'as, and Matula in 1976.
    In this paper we prove an asymptotically best possible result for induced matchings by showing that if  $C/n\le p \le 0.99$ for some large constant $C$, then $G(n,p)$ contains an induced matching of order approximately $2\log_q(np)$,
    where $q= \frac{1}{1-p}$.
\end{abstract}

\section{Introduction}

Let $G(n,p)$ denote the binomial random graph on vertex set
$[n]:= \{1,2,\ldots,n\}$, where each edge is included independently with probability $p$. The following classic question has been extensively studied in the theory of random graphs:
given a large graph $H$, does $G(n,p)$ contain a copy of $H$ as an induced subgraph with high probability (abbreviated to whp, meaning with probability tending to $1$ as $n$ tends to infinity)? One of the first instances of this problem is determining the independence number of $G(n,p)$, i.e.\ when $H=H_k$ is an empty graph on $k$ vertices,
how large can $k$ be such that $H$ is an induced subgraph of $G(n,p)$. The study of this particular instance
dates back to 1976 when Bollob\'as and Erd\H{o}s~\cite{bollobas1976cliques}
and Matula~\cite{matula1976largest} showed that the independence number of $G(n,p)$ for constant $p$ is asymptotically
$2\log_q(np)$, where $q=1/(1-p)$. A simple first moment argument shows that the size of this empty subgraph is asymptotically largest possible.
Frieze~\cite{frieze1990independence} extended this result to the sparse regime, when $p=c/n$ for a large enough constant $c$, with the same expression $2\log_q(np)$ for the asymptotic size of the largest independent set. Indeed, it can be shown that the same result holds for every $p=p(n)\gg 1/n$ (see e.g.~\cite{janson2011random}).

\par Another classic result which deals with non-empty induced subgraphs of $G(n,p)$ 
is due to Erd\H{o}s and Palka~\cite{erdos1983trees}. They showed that whp the largest induced tree in $G(n,p)$ is asymptotically of size $2\log_q(np)$ if $p$ is a constant. Furthermore, they conjectured that the largest induced tree in the sparse regime (when $p=c/n$ for a large constant $c$) is of linear size. Frieze and Jackson \cite{frieze1987large2}, Ku{\v{c}}era and R{\"o}dl~\cite{kuvcera1987large}, {\L}uczak and Palka~\cite{luczak1988maximal}, and de la Vega~\cite{de1986induced} independently proved this conjecture. Subsequently, 
for $p=c\ln n/n$, where $\ln$ denotes the natural logarithm,
Palka and Ruci{\'n}ski~\cite{palka1986order} showed that the largest induced tree is of size between $\log_q(np)$ and $2\log_q(np)$.
Finally, de la Vega~\cite{de1996largest} showed that for $p=c/n$ the largest induced tree has size asymptotically $2\log_q(np)\sim 2\frac{\ln c}{c}n$. Although, de la Vega proves his result only for $p=c/n$, one can check that his ideas extend to any larger $p$ as well. 

Note that the fact that $G(n,p)$ contains large induced trees does not give much information on what these trees look like. Therefore, a natural question is whether a given {\em fixed} large tree is an induced subgraph of $G(n,p)$. The first steps in this direction were made by Frieze and Jackson~\cite{frieze1987large} and Suen~\cite{suen1992large}, who showed that the length of the longest induced path in $G(n,c/n)$ is linear for $c$ large enough. {\L}uczak~\cite{luczak1993size} improved upon their results, by proving that the length of the longest induced path in $G(n,c/n)$ is between $\log_q(np)$ and $2\log_q(np)$. For constant $p$, Ruci{\'n}ski~\cite{rucinski1987induced} showed that the longest induced path in $G(n,p)$ is of length asymptotically $2\log_q(np)$, which was later extended to all $p\geq n^{-1/2}(\ln n)^2$
by Dutta and Subramanian in~\cite{dutta2018induced}.

When $p$ is small, it is natural to require some restriction
on the maximum degree $\Delta(H)$ of $H$, simply because
whp $G(n,p)$ does not have any vertices of large degree.
In particular, when $p=c/n$ then a natural case to study is when $H$ is a tree of bounded degree, i.e.\ a tree with $\Delta(H)<d$ for some constant $d$. The second author \cite{draganic2020large} proved that for $n^{-1/2}(\ln n)^2 \le p \le 0.99$ and every fixed bounded degree tree $T$ of size asymptotically $2\log_q(np)$, $G(n,p)$ contains $T$ with high probability.
On the other hand, in sparser random graphs very little is known about induced
copies of general large bounded degree trees.

Another natural class of induced subgraphs to look for in $G(n,p)$ is induced matchings, which are in some sense an interpolation between independent sets and trees. For constant $p$, it has been shown by Clark~\cite{clark2001strong} that whp $G(n,p)$ contains induced matchings with $(2\pm o(1))\log_q(np)$ vertices. 

In this paper we establish the following result on induced matchings, stating that the largest induced matching in $G(n,p)$ contains roughly $2\log_q(np)$ vertices, which is an asymptotically optimal result.\\
\begin{thm}\label{main}
For all $\varepsilon_0>0$ there exists $C=C(\varepsilon_0)>0$ such that 
whp the largest induced matching in $G(n,p)$ contains ${(1 \pm \varepsilon_0)}{\log_q (np) }$ edges, where $q=\frac{1}{1-p}$, whenever $\frac{C}{n}\leq p\leq 0.99$.
\end{thm}

As described above, the size of largest independent sets in random graphs is well understood,
and there are also several known results on the size of the largest induced tree for various regimes of $p$,
as well as for the largest induced matching when $p$ is constant.
However, for a \emph{fixed} (i.e.\ previously specified) induced bounded degree tree in the sparse regime, we know very little. Our result for induced matchings is the first step in understanding this problem.

For $p>\frac{(\ln n)^2}{\sqrt{n}}$ the aforementioned results for independent sets, paths, bounded degree trees, and matchings can be proved using the second moment method.  On the other hand, the vanilla second moment calculations break down roughly when $p\sim\frac{1}{\sqrt{n}}$ and for smaller $p$ all of the aforementioned problems become significantly harder.

Our proof relies on two main ingredients --- the second moment method and Talagrand's inequality. Although the second moment method on its own is of little use in sparse regimes, in combination with strong concentration bounds (such as Talagrand's inequality) it yields a nice tool which can be very powerful, as was already demonstrated by Frieze in~\cite{frieze1990independence}.

\section{Large induced matching: proof of Theorem~\ref{main}}

Note that for constant $p$, \Cref{main} is exactly the aforementioned result of Clark~\cite{clark2001strong}.
In this section we will prove \Cref{main} for $p\le\frac{1}{(\ln n)^3}$. For  $\frac{1}{(\ln n)^3} \le p = o(1)$ it is enough to use just a standard second moment argument with Chebyshev's inequality,
whose calculations are much simpler than those required in the sparser case --- we include a proof in
Appendix~\ref{app:secondmoment} for completeness.

Throughout the paper we will use the standard Landau notations 
$o(\cdot),O(\cdot),\Theta(\cdot),\Omega(\cdot),\omega(\cdot)$.
When not otherwise explicitly stated, the asymptotics in this
notation are with respect to $n$. We will also use this notation
with asymptotics with respect to $c$, in which case we add $c$
explicitly to the notation. For example, $f= o_c(g)$ means that
$|f|/|g| \xrightarrow{c \to \infty} 0$. We will omit
floors and ceilings when these do not significantly affect the argument.

Given a graph $H$, we denote by $\M(H)$ the size of (i.e.\ the number of edges in)
the largest induced matching in $H$.
We want to show that $\M(G(n,p)) \geq {(1-\varepsilon_0)}{\log_q (np) }$ with high probability. Note that if $p=o(1)$ then this is asymptotically equal to $(1-\varepsilon_0)\frac{\ln (np)}{p}$, so we will work with the latter expression to ease notation.

We use the notation $a \gg b$ to mean that for some implicit function $f:\mathbb{R}\to \mathbb{R}$, 
we have $a \ge f(b)$. We will not determine the function $f$ that we require explicitly, although it could be deduced from a careful analysis of the calculations.
For the rest of the paper, we
fix the following parameters.
Let $\eps_0>0$, let $\eps := \frac{\eps_0}{3}$ and let
$p=c/n$ where $c=c(n)$ is a function of $n$ satisfying
$\frac{n}{(\ln n)^3} \ge c \gg \varepsilon^{-1}$. 
Let
\begin{equation}\label{size}
k:=\frac{(1-\eps)\ln c}{c}n.
\end{equation}

Let $G\sim G(n,p)$ and let $Y_k$ be the random variable
which counts the number of induced matchings of size $k$ in $G$. 
We will prove two lemmas which directly imply our theorem.
The first lemma tells us that $\M(G)$ is well concentrated
in the sense that it cannot have both upper and lower tail having
large probability.
\begin{lem}\label{probability1}
\begin{equation*}
    \pr\left(\M(G)\leq k- \varepsilon\frac{\ln c}{c}n\right)
    \cdot \pr\Big(\M(G)\geq k\Big) \leq \exp\left(-\frac{2n}{c}\right).
\end{equation*}
\end{lem}
We will prove \Cref{probability1} in \Cref{sec:azuma} using 
an application of Talagrand's inequality.

The second lemma gives a rather weak estimate on the probability that a large matching occurs,
but which in combination with \Cref{probability1} is enough to show \Cref{main}.
\begin{lem}\label{probability2}
\begin{equation*}
    \pr\left(Y_{k}>0\right)\geq \exp\left(-\frac{n}{c}\right).
\end{equation*}
\end{lem}
The proof of \Cref{probability2} is based on the second moment method using the Paley-Zygmund inequality
and appears in \Cref{sec:secondmoment}.

Now we show how \Cref{main} follows from these two lemmas. We begin with a simple
first moment calculation. Analogously to $Y_k$, for any $r\in \mathbb{N}$ let $Y_r$ denote the number of
induced matchings of size $r$ in $G$.

\begin{claim}\label{claim:firstmoment}
For any positive integer $r$, we have
$$\EE[Y_r]=\binom{n}{2r}\frac{(2r)!}{r!2^r}p^r(1-p)^{\binom{2r}{2}-r}.$$
\end{claim}

\begin{proof}
There are $\binom{n}{2r}(2r-1)!!=\binom{n}{2r}\frac{(2r)!}{r!2^r}$
possible matchings $M$ of size $r$. Furthermore, in order for
$M$ to form an induced matching in $G(n,p)$, all $r$ edges must be present,
and furthermore the remaining $\binom{2r}{2}-r$ pairs in $V(M)$ may not be
edges of $G(n,p)$, which occurs with probability $p^r (1-p)^{\binom{2r}{2}-r}$.
Combining these two terms gives the claim.
\end{proof}

\begin{proof}[ of \Cref{main}]
We first prove the upper bound using the first moment method.
Setting
$$
r=(1+\eps_0)\log_q(np) = (1+\eps_0) \frac{\ln (np)}{-\ln(1-p)} = \Theta(1) \frac{\ln(np)}{p}
$$
(where the last equality follows since $p\le 0.99$),
let us observe that
\Cref{claim:firstmoment} gives
\begin{align*}
\EE[Y_r] = \binom{n}{2r}\frac{(2r)!}{r!2^r}p^r(1-p)^{\binom{2r}{2}-r}
& \le \frac{n^{2r}}{(2r)!}\frac{(2r)!}{(2r/e)^r}p^r (1-p)^{r(2r-2)}\\
& = \left(\Theta(1) \frac{n^2p(1-p)^{2r}}{r}\right)^r\\
& = \left(\Theta(1) \frac{n^2p(np)^{-2(1+\eps_0)}}{(\ln (np))/p}\right)^r\\
& = \left(\Theta(1) \frac{(np)^{-2\eps_0}}{\ln (np)}\right)^r.
\end{align*}
Now if $p=\Theta\left(n^{-1}\right)$, then $r\ge \log_q(np) = \frac{\ln(np)}{-\ln(1-p)}\to \infty$,
and recalling that $p \ge \frac{C}{n}$ for some $C=C(\eps_0)$ sufficiently large,
we have
$$
\EE[Y_r] \le \left(\Theta(1) \frac{C^{-2\eps_0}}{\ln (C)}\right)^r \le \left(\frac{1}{2}\right)^r =o(1).
$$
On the other hand, if $p= \omega\left(n^{-1}\right)$, then we have
$$
\EE[Y_r] \le \left(o(1)\right)^r = o(1).
$$
In both cases an application of Markov's inequality shows that whp $Y_r=0$, as required.

To prove the lower bound, as mentioned previously we assume that $p<\frac{1}{(\ln n)^3}$.

Recalling that for $p=o(1)$ we have
$$
\log_q(np) = \frac{\ln (np)}{\ln \left(\frac{1}{1-p}\right)} = (1+o(1))\frac{\ln(np)}{p},
$$
the two lemmas together imply that
\begin{align*}
\pr\left(\M(G)\le (1-\varepsilon_0) \log_q(np)\right)
& \le \pr\left(\M(G)\le \left(1-\frac{2\varepsilon_0}{3}\right)\frac{\ln c}{c}n\right)\\
&= \pr\Big(\M(G)\leq k- \varepsilon\frac{\ln c}{c}n\Big)\\
 & \stackrel{\mbox{{\scriptsize L.\ref{probability1}}}}{\le}
 \frac{\exp\left(-\frac{2n}{c}\right)}{\pr\left(\M(G)\geq k\right)}
 = \frac{\exp\left(-\frac{2n}{c}\right)}{\pr\left (Y_k>0\right)}
 \stackrel{\mbox{{\scriptsize L.\ref{probability2}}}}{\le} \exp\left(-\frac{n}{c}\right),
\end{align*}
which tends to zero as required.
\end{proof}

\section{Concentration using Talagrand's inequality: proof of  \Cref{probability1}}\label{sec:azuma}
Talagrand's inequality is a useful tool to show that under certain conditions a random variable is tightly concentrated. We will use it in the form which appears in \cite{alon2004probabilistic}.

\begin{defn}
Let $\Omega=\prod_{i=1}^{n}\Omega_i$ be a product of probability spaces such that $\Omega$ has the product measure.
Let $g: \Omega\to \mathbb{R}$ and $f:\mathbb{N}\to \mathbb{N}$ be functions.
\begin{itemize}
    \item We say that $g$ is \emph{Lipschitz} if $|g(x)-g(y)|\leq 1$ for every $x,y\in \Omega$ which differ in at most one coordinate.

\item We say that $g$ is \emph{$f$-certifiable} if for any $x\in \Omega$ and $m\in \mathbb{N}$ such that $g(x)\ge m$,
there exists a set of coordinates $I\subset [n]$ with $|I|\leq f(m)$ such that each $y \in \Omega$ which agrees with $x$ on $I$ also satisfies $g(y)\geq m$.
\end{itemize}
\end{defn}

\begin{thm}[Talagrand]\label{Talagrand}
Let $X$ be a Lipschitz random variable on $\Omega$ which is $f$-certifiable. Then for all $\lambda>0$ and $b\in \mathbb{N}$ it holds that:
\[
    \pr\left(X<b-\lambda \sqrt{f(b)}\right)\cdot \pr\left(X\geq b\right) \leq \exp\left(-\frac{\lambda^2}{4}\right).
\]
\end{thm}

In order to prove \Cref{probability1} we will 
regard $G(n,p)$ as a product of $n-1$ probability spaces $Z_i$, $i\in [n-1]$. 
Recall that $G(n,p)$ is a graph on vertex set $[n]$.
Each $Z_i$ picks uniformly at random a subset of $[i]$ of size $\mathrm{Bi}(i,p)$ -- these are the neighbours of vertex $i+1$ within $[i]$. It is easy to see that this is equivalent to $G(n,p)$. 

\begin{proof}[ of \Cref{probability1}]
Let $G\sim G(n,p)$. Then the random variable $\M=\M(G)$ is Lipschitz. Indeed, note that by changing a particular $Z_i$ one can change $\M$ only by at most $1$,
since if $G'$ is obtained from $G$ by changing some edges adjacent to vertex $i+1$
and if $M$ is a largest matching in $G$, then certainly $M'$, which is obtained
from $M$ by deleting $i+1$ and its partner if it lies in $M$, is a matching in $G'$,
and therefore $\M(G')\ge \M(G)-1$. By symmetry also $\M(G')\le \M(G)+1$.

Furthermore, if $\M\geq m$ then there exists a set $S\subset V(G)$ of $2m$ vertices which induces a matching of size $m$. By fixing $Z_{i-1}$ for each $i\in S$ (where we interpret $Z_0$ as an empty random variable), changing other coordinates can only increase the largest induced matching in $G$. Therefore $\M$ is $f$-certifiable with $f(m)=2m$. This means that we can apply \Cref{Talagrand} with parameters $b=k$ and $\lambda =\varepsilon\frac{\ln c}{c\sqrt{2k}}n$, and observing that $\lambda \sqrt{f(b)} = \varepsilon \frac{\ln c}{c}n$ we obtain:
\begin{align*}
     \pr\left(\M<k-\varepsilon \frac{\ln c}{c}n\right)\cdot P\left(\M\geq k\right)
     \leq \exp \left(-\frac{\lambda^2}{4}\right)
     & = \exp\left(-\frac{\varepsilon^2(\ln c)^2n^2}{8c^2k}\right)\\
     & \stackrel{\eqref{size}}{=}\exp\left(-\frac{\varepsilon^2}{8(1-\varepsilon)}\frac{\ln c}{c}n\right)\leq 
        \exp \left (-\frac{2n}{c}\right)
\end{align*}
which completes the proof.
\end{proof}

\section{Second moment method: proof of  \Cref{probability2}}\label{sec:secondmoment}

\par Consider the family $\{M_i\mid i\in I\}$ of all sets of $k$ unordered disjoint pairs of vertices in $G$,
i.e.\ the family of possible matchings of size $k$. For $i\in I$ let $X_i$
be the indicator random variable which indicates that the pairs in $M_i$ form an induced matching in $G$.
In particular, it holds that $Y_k=\sum_{i\in I}X_i$.
The main difficulty is to prove the following.
\begin{lem}\label{lem:secondmoment}
\begin{equation}\label{ToProve}
    \frac{\sum_{i\in I}\EE[X_i|X_1=1]}{\EE[Y_k]}\leq \exp\left(\frac{n}{c}\right).
\end{equation}
\end{lem}{}

\begin{proof}[ of \Cref{probability2}]

We will use the well-known inequality
\begin{equation}\label{eq:pz}
\pr(Y_k>0) \ge \frac{\EE[Y_k]^2}{\EE[Y_k^2]},
\end{equation}
which can be deduced, for example, as a special case of the Paley-Zygmund inequality.
We now observe that
\begin{align*}
    \EE[Y_k^2]&= \sum_{i,j\in I}\EE[X_iX_j]= \sum_{j\in I}\sum_{i\in I}\EE[X_i|X_j=1]\EE[X_j]\\
    &= \sum_{j\in I}\sum_{i\in I}\EE[X_i|X_1=1]\EE[X_1]\\
    &=|I|\cdot \EE[X_1]\sum_{i\in I}\EE[X_i|X_1=1] 
    = \EE[Y_k]\sum_{i\in I}\EE[X_i|X_1=1].
\end{align*}
Therefore, using  \Cref{lem:secondmoment}, we obtain
\begin{align*}
    \pr(Y_k>0) \stackrel{\eqref{eq:pz}}{\ge} \left(\frac{\EE[Y_k^2]}{\EE[Y_k]^2}\right)^{-1}
    & = \left(\frac{\sum_{i\in I}E[X_i|X_1=1]}{\EE[Y_k]}\right)^{-1}
     \stackrel{\eqref{ToProve}}{\ge} \exp\left(-\frac{n}{c}\right)
\end{align*}
as required.
\end{proof}

To prove \Cref{lem:secondmoment}, we will consider the numerator and denominator
separately in the next subsection.

\subsection{Conditional expectation}
In this subsection we will give an explicit expression for
the numerator in the left-hand side of~\eqref{ToProve};
for the denominator, we will just substitute the expression
from \Cref{claim:firstmoment}.
We begin with the following lemma,
whose proof forms the main part of this subsection.

\begin{lem}\label{lem:explicitsum}
Given $\ell,s\in \mathbb{N}_0$ with $\ell+s \le k$, let us define
$a_{\ell,s} =a_{\ell,s}(n,p,k,c,\eps)$ by
$$
a_{\ell,s} :=
2^{\ell+2s-k}\frac{k!}{\ell!s!((k-\ell-s)!)^2}\cdot\frac{(n-2k)!}{(n-4k+2\ell+s)!}
p^k \left(\frac{k}{(1-\eps)\ln c}\right)^{\ell}
    (1-p)^{\binom{2k}{2}-k-\left(\binom{2\ell+s}{2}-\ell\right)}.
    $$
Then
\begin{align*}
\sum_{i\in I}\EE[X_i|X_1=1] = \sum_{\ell=0}^k \sum_{s=0}^{k-\ell} a_{\ell,s}.
\end{align*}
\end{lem}

We first determine which $i\in I$ give a non-zero contribution
to $\sum_{i\in I}\EE[X_i|X_1=1]$.

\begin{defn}\label{compatible}
We say that $M_i$ is \emph{compatible} with $M_1$ if
\begin{enumerate}
    \item $M_i$ contains no pair $\{u,v\}$ whose vertices lie in different pairs of $M_1$;
    \item $M_1$ contains no pair $\{u,v\}$ whose vertices lie in different pairs of $M_i$;
\end{enumerate}
\end{defn}

As a consequence of this definition, we observe that we can classify the pairs of $M_i$
into types according to their intersection with $M_1$. More precisely,
for any $i\in I$, denote by $V(M_i)$ the set of vertices contained in some pair of $M_i$.
Then we have the following.

\begin{nremark}\label{rem:compatible}
If $M_i$ is compatible with $M_1$, then $M_i$ only contains the following types of pairs (see \Cref{fig:TypesOfPairs}):
\begin{enumerate}[(A)]
    \item \label{pairs:2} Pairs from $M_1$;
    \item \label{pairs:1} Pairs which contain one vertex from $V(M_1)$ and one vertex outside $V(M_1)$;
    \item \label{pairs:0} Pairs with no vertex in $V(M_1)$.
\end{enumerate}
Furthermore the only pairs of $M_1$ whose endpoints both lie in $V(M_i)$
are also pairs of $M_i$ (or equivalently, conditions (A), (B) and (C)
also hold with $M_1$ and $M_i$ switched).
\end{nremark}

\begin{figure}
    \centering
    \begin{tikzpicture}
     \node[scale=2.2] (c) at  (-5,0)  {$M_i$};
       \node[scale=2.2] (c) at  (9.5 ,0)  {$M_1$}; 
    \node[scale=1.5] (A) at (2.25,-2.2) {(A)};    
    \node[scale=1.5] (B) at (1,2.4) {(B)};    
    \node[scale=1.5] (C) at (-2,-2.2) {(C)};    
    
         \draw[fill=lightgray, opacity=0.4]
         (0,0) ellipse (4 and 2);
         \draw[fill=lightgray, opacity=0.4]
         (4.5,0) ellipse (4 and 2);
         
          \tikzstyle{every node}=[circle, draw, fill=black!50,
                        inner sep=0pt, minimum width=3pt]
        \foreach \i in {1,2,3,4,5,6}
         {
            \node[] (v) at  (-3+\i*0.3,-0.7){};
            \node[] (u) at   (-3+\i*0.3,+0.7){};
            \draw[color=black, dotted, line width = 1pt] (u) to (v);
        }
        
         \foreach \i in {0,1,2}
         {
         
            \node[] (v) at (0.1,0.3+0.4*\i){};
            \node[] (u) at (2.25, 0.3+0.4*\i){};
            \draw[color=black,dashed, line width=1pt] (u)--(v);
            \node[] (v) at (4.4,0.3+0.4*\i){};
            \node[] (u) at (2.25, 0.3+0.4*\i){};
            \draw[color=black, line width=3pt] (u)--(v);
         }
        
         \foreach \i in {2,3,4,5,6}
         {
            \node[] (v) at  (1+\i*0.3,-1.3){};
            \node[] (u) at   (1+\i*0.3,0){};
            \draw[color=black, line width = 1pt] (u) to (v);
        }
        
        \foreach \i in {1,2,3,4,5,6}
         {
            \node[] (v) at  (7.5-\i*0.3,-0.7){};
            \node[] (u) at   (7.5-\i*0.3,+0.7){};
            \draw[color=black, line width = 3pt] (u) to (v);
        }
        
    \end{tikzpicture}
    \caption{Type (A),(B) and (C) pairs illustrated
    as plain, dashed and dotted lines,
    respectively. Pairs which are in $M_1$ but not in $M_i$ are
    thick.}
    \label{fig:TypesOfPairs}
\end{figure}
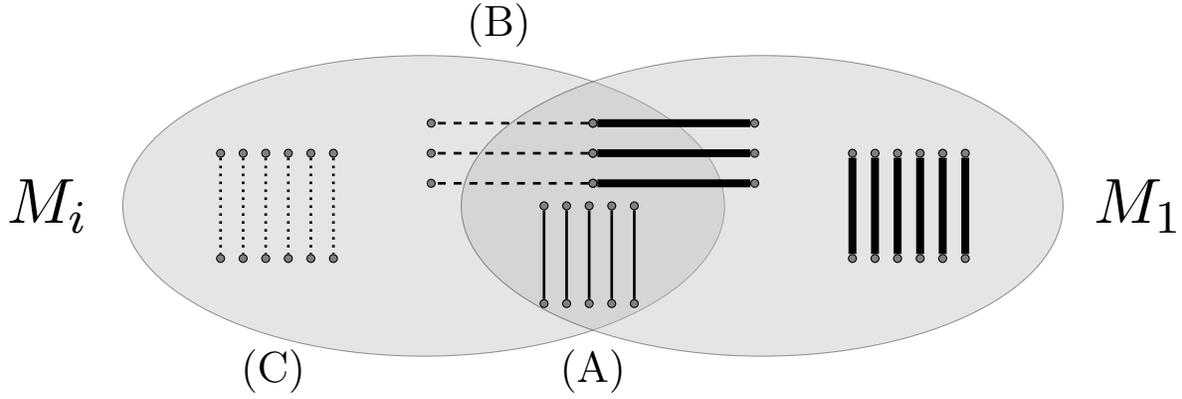

\begin{claim}
If $M_i$ is not compatible with $M_1$, then $\EE[X_i|X_1=1]= 0$.
\end{claim}

\begin{proof}
If $M_i$ violates the first condition of compatibility with $M_1$, i.e.\ if
$M_i$ contains a pair of vertices $\{u,v\}$ from different pairs in $M_1$,
then under the assumption that $X_1=1$, since $M_1$ is an \emph{induced} matching,
certainly $\{u,v\}$ is not an edge in $G$, hence $X_i=0$. By symmetry, if the second
condition is violated, it is also not possible that $X_i=X_1=1$.
\end{proof}

We may therefore restrict our attention to matchings $M_i$ that are compatible with $M_1$.
Next we define an equivalence relation on $I$
(or more precisely on the subset of those $i\in I$ such that $M_i$ is compatible with $M_1$)
such that
for each $i$ in the same equivalence class, the expression $\EE[X_i|X_1=1]$ is the same.

\begin{defn}
Define $I(\ell,s)$ to be the set of $i\in I$ such that $M_i$ is compatible with $M_1$
and has exactly $\ell$ pairs of vertices which are of type~\eqref{pairs:2}
and $s$ pairs of type~\eqref{pairs:1} (see \Cref{rem:compatible}).
\end{defn}

Thus we have
\begin{equation}\label{eq:partitionedsum}
    \sum_{i\in I}\EE[X_i|X_1=1] = \sum_{\ell,s} \sum_{i\in I(\ell,s)}\EE[X_i|X_1=1].
\end{equation}

Now for fixed $\ell,s$, we can handle the conditional expectation with the following claim.

\begin{claim}\label{claim:conditionalProbability}
For $i\in I(\ell,s)$, we have
\begin{equation}\label{conditionalProbability}
    \EE[X_i|X_1=1]=p^k \left(\frac{k}{(1-\eps)\ln c}\right)^{\ell}
    (1-p)^{\binom{2k}{2}-k-\left(\binom{2\ell+s}{2}-\ell\right)}.
\end{equation}
\end{claim}

\begin{proof}
Conditioning on the event that $X_1=1$,
we observe that the $\ell$ pairs of type~\eqref{pairs:2} are already automatically
present as edges because they are pairs of $M_1$, but we require a further
$k-\ell$ pairs to be present as edges, which occurs with probability
$p^{k-\ell} = p^k \left(\frac{k}{(1-\eps)\ln c}\right)^{\ell}$.
Furthermore, we require all of the
$\binom{2k}{2}-k$ pairs which lie within $V(M_i)$ but are not in the
matching $M_i$ to be non-edges of $G(n,p)$; however, those pairs that lie
inside $V(M_1)$, of which there are $\binom{2\ell+s}{2}-\ell$,
are already guaranteed to be non-edges by the conditioning
on $X_1=1$. Thus the probability that all appropriate pairs are
non-edges is $(1-p)^{\binom{2k}{2}-k - \left(\binom{2\ell+s}{2}-\ell\right)}$.
Multiplying the two terms together, we obtain the statement of the claim.
\end{proof}

The sum over $I(\ell,s)$ in~\eqref{eq:partitionedsum} is dealt with using the following result.

\begin{claim}\label{claim:Counting}
Let $\ell,s \in \mathbb{N}_0$.
\begin{enumerate}[(i)]
    \item If $\ell+s >k$, then $|I(\ell,s)|=0$.
    \item If $\ell+s \le k$, then \begin{equation}\label{Counting}
    |I(\ell,s)|=2^{\ell+2s-k}\frac{k!}{\ell!s!((k-\ell-s)!)^2}\cdot\frac{(n-2k)!}{(n-4k+2\ell+s)!}.
\end{equation}
\end{enumerate}
\end{claim}

\begin{proof}
The first statement is clear, since in order for $i$ to lie in $I(\ell,s)$, the matching $M_i$ must contain $\ell$ pairs of type~\eqref{pairs:2} and $s$ pairs of type~\eqref{pairs:1}, but $k$ pairs in total.

For the second statement,
observe that there are $\binom{k}{\ell}$ ways of choosing
the $\ell$ pairs of type~\eqref{pairs:2}.
We subsequently choose the $s$ endpoints within $M_1$ of pairs of
type~\eqref{pairs:1}, for which there are $\binom{k-\ell}{s}2^s$ possible choices.
For the other endpoints of these $s$ pairs, we have $\binom{n-2k}{s}$ choices
for the vertices outside $V(M_1)$, and $s!$ ways of matching
them with the $s$ endpoints already chosen within $V(M_1)$.
Finally, we have
$\binom{n-2k-s}{2k-2\ell-2s}$ ways of choosing the vertices for the remaining
$k-\ell-s$ pairs of type~\eqref{pairs:0} (while avoiding $V(M_1)$ and the further $s$
vertices chosen for pairs of type~\eqref{pairs:1}),
and $(2k-2\ell-2s-1)!!= \frac{(2k-2\ell-2s)!}{(k-\ell-s)!2^{k-\ell-s}}$
ways of choosing a perfect matching on these vertices. Collecting all these terms gives
\begin{align*}
|I(\ell,s)| & = \binom{k}{\ell}\binom{k-\ell}{s}2^s\binom{n-2k}{s}s!\binom{n-2k-s}{2k-2\ell-2s} \frac{(2k-2\ell-2s)!}{(k-\ell-s)!2^{k-\ell-s}} \\
& = \frac{k!}{\ell!s!(k-\ell-s)!}2^s \frac{(n-2k)!}{(n-4k+2\ell+s)!(k-\ell-s)!2^{k-\ell-s}} \\
& = 2^{\ell+2s-k}\frac{k!}{\ell!s!((k-\ell-s)!)^2}\cdot\frac{(n-2k)!}{(n-4k+2\ell+s)!}
\end{align*}
as claimed.
\end{proof}{}

We can combine the two previous claims to prove \Cref{lem:explicitsum}.

\begin{proof}[ of \Cref{lem:explicitsum}]
Observe that the statement of \Cref{lem:explicitsum} follows directly
by applying~\eqref{eq:partitionedsum} and \Cref{claim:conditionalProbability,claim:Counting}.
\end{proof}

Now let us define $b_{\ell,s} = b_{\ell,s}(n,p,k,c,\eps)$ by
\begin{align}\label{eq:bkls}
    b_{\ell,s} & := a_{\ell,s} \frac{k!2^k}{(2k)!\binom{n}{2k}}p^{-k} (1-p)^{-\binom{2k}{2}+k}\nonumber\\
    & = 2^{\ell+2s}\frac{(k!)^2}{\ell!s!((k-\ell-s)!)^2}\cdot\frac{((n-2k)!)^2}{n!(n-4k+2\ell+s)!}\left(\frac{k}{(1-\eps)\ln c}\right)^{\ell}(1-p)^{-\binom{2\ell+s}{2}+\ell}.
\end{align}
Then we have the following immediate consequence of \Cref{lem:explicitsum} and
\Cref{claim:firstmoment} (applied with $r=k$).

\begin{cor}\label{cor:explicitsum}
$$\frac{\sum_{i\in I}\EE[X_i|X_1=1]}{\EE[Y_k]} = \sum_{\ell=0}^k \sum_{s=0}^{k-\ell} b_{\ell,s}.
$$
\end{cor}{}

\subsection{Analysing the summands}

Given \Cref{cor:explicitsum}, we will aim to bound each of the summands $b_{\ell,s}$,
which is the goal of this subsection.

\begin{lem}\label{lem:boundbkls}
For any $\ell,s \in \mathbb{N}_0$ satisfying $\ell+s\le k$, we have
$$
b_{\ell,s} \le \exp\left(\frac{n}{2c}\right) .
$$
\end{lem}{}

In the main argument we will use some approximations which are not well-defined
if any one of $s,\ell$ or $k-\ell-s$ is $0$,
so we first deal with such terms by comparing them to others.

\begin{claim}\label{prop:boundbratios}
For any $\ell,s \in [k]_0$ with $\ell+s\le k$, we have
$$
b_{\ell,s} \le n^{9} \max_{\substack{1\le i \le k-2\\ 1\le j \le k-i-1}}b_{i,j}.
$$
\end{claim}

To prove this claim, it suffices to compare
terms on the ``boundary'', i.e.\ when $\ell=0$,
when $s=0$ or when $\ell+s =k$, with other terms.
For example, it is elementary to check that
for $0 \le \ell \le k-1$, we have
$ b_{\ell,0} \le n\cdot b_{\ell,1}$.
In other words, we may ``move'' from $s=0$
to $s=1$ at a cost of a multiplicative factor $n$.
By making at most three such moves,
each at a cost of at most $n^3$,
we can reach the interior of the region,
i.e.\ $1 \le \ell,s \le k-1$ and $\ell+s \le k-1$,
from any point on the boundary.
The full proof is included in Appendix~\ref{app:ratios} for completeness.

In particular, \Cref{prop:boundbratios} allows us to restrict our attention to the
case when $\ell,s,k-\ell-s \neq 0$.
\begin{prop} \label{prop:boundbklspositive}
If $1\le \ell,s \in \mathbb{N}$ and $\ell+s \le k-1$,
then
$$
b_{\ell,s} \le \exp\left(\frac{n}{3c}\right) .
$$
\end{prop}{}

\begin{proof}

We will write
$P(n)$ for any term that is
polynomial in $n$ (i.e.\ there exists a polynomial $Q$ such that $\frac{1}{Q(n)} \le P(n) \le Q(n)$
for sufficiently large $n$).
Since $\ell,s,k-\ell-s\neq 0$, we may apply Stirling's approximation
to various terms in~\eqref{eq:bkls} and obtain
\begin{align*}
    b_{\ell,s} & = 2^{\ell+2s}
    \frac{P(k)\left(\frac{k}{e}\right)^{2k}}
    {\left(\frac{\ell}{e}\right)^\ell\left(\frac{s}{e}\right)^s\left(\frac{k-\ell-s}{e}\right)^{2(k-\ell-s)}}
    \cdot
    \frac{P(n)\left(\frac{n-2k}{e}\right)^{2(n-2k)}}
    {\left(\frac{n}{e}\right)^{n}\left(\frac{n-4k+2\ell+s}{e}\right)^{n-4k+2\ell+s}}
    \left(\frac{k}{(1-\eps)\ln c}\right)^{\ell}(1-p)^{-\binom{2\ell+s}{2}+\ell}
    \\
    & = P(n) 2^{\ell+2s}
    \frac{e^\ell k^{2k+\ell}}{\ell^\ell s^s (k-\ell-s)^{2(k-\ell-s)}} \cdot
    \frac{(n-2k)^{2(n-2k)}}{n^n (n-4k+2\ell+s)^{n-4k+2\ell+s}}
    \left(\frac{1}{(1-\eps)\ln c}\right)^{\ell}
    (1-p)^{-\binom{2\ell+s}{2}+\ell}\\
    & = P(n) \frac{\left(\frac{k}{\ell}\right)^\ell\left(\frac{k}{s}\right)^s \left(\frac{k}{k-\ell-s}\right)^{2(k-\ell-s)}\left(\frac{n-2k}{n}\right)^{n}\left(\frac{n-2k}{n-4k+2\ell+s}\right)^{n-4k+2\ell+s}}{\left(\frac{n-2k}{k}\right)^{2\ell+s}}
    2^{2s}\left(\frac{2e(1-p)}{(1-\varepsilon)\ln c}\right)^\ell(1-p)^{-\binom{2\ell+s}{2}}\\
    & = P(n) \frac{e^{\ell\ln(k/\ell)+s\ln(k/s)}(1+\frac{\ell+s}{k-\ell-s})^{2(k-\ell-s)}(1-\frac{2k}{n})^n
    (1+\frac{2k-2\ell-s}{n-4k+2\ell+s})^{n-4k+2\ell+s}}{(\frac{n-2k}{k})^{2\ell+s}(1-p)^{\binom{2\ell+s}{2}}}
    2^{2s}\left(\frac{2e(1-p)}{(1-\varepsilon)\ln c}\right)^{\ell}.\\
\end{align*}
Let us observe that $\frac{2e(1-p)}{(1-\eps)\ln c}\le 1$ for sufficiently large $c$,
so for an upper bound we may ignore the final term.
For what remains, we will use the inequalities $e^{-x-x^2}\leq 1-x \leq e^{-x-x^2/2}$
and $1+x \leq e^{x-x^2/2+x^3/3}\leq e^{x}$ which hold for all positive $x<1/2$,
and indeed the inequality $1+x \le e^x$ holds for any real number $x$.
Thus observing that $\ln(P(n))= \Theta(\ln n)$, we obtain
\begin{align}\label{exponent}
    \ln(b_{\ell,s}) & \le \Theta(\ln n) +
    \ell\ln\left(\frac{k}{\ell}\right)+s\ln\left(\frac{k}{s}\right)
    + 2(\ell+s) -\left(2k+\frac{2k^2}{n}\right) \nonumber \\
    & \hspace{2cm} + (2k-2\ell-s)-\frac{(2k-2\ell-s)^2}{2(n-4k+2\ell+s)} 
    + \frac{(2k-2\ell-s)^3}{3(n-4k+2\ell+s)^2}
    \nonumber \\
    & \hspace{2cm}
    +2s\ln 2
    -\ln\left(\frac{n-2k}{k}\right)(2\ell+s) 
    -\left(-\frac{c}{n}-\frac{c^2}{n^2}\right)\frac{(2\ell+s)^2}{2}
     \nonumber \\
    & \le \Theta(\ln n) +
    \ell\ln\left(\frac{k}{\ell}\right)+ s\ln\left(\frac{k}{s}\right)+s -\frac{2k^2}{n}
    -\frac{(2k-2\ell-s)^2}{2n} + \frac{3k^3}{n^2} \nonumber \\
    & \hspace{2cm} + 2s\ln 2 -\ln\left(\frac{n-2k}{k}\right)(2\ell+s)
    +\frac{c}{n}\left(1+\frac{\eps}{2}\right)\frac{(2\ell+s)^2}{2}  
    \nonumber \\
    & = F + o_c\left(\frac{1}{c}\right)n ,
\end{align}
where we define
\begin{align}\label{eq:fnkls}
    F=F(n,k,c,\ell,s) & := \ell\ln\left(\frac{k}{\ell}\right)+ s\ln\left(\frac{k}{s}\right)+s + 2s\ln 2
    +\frac{c}{n}\left(1+\frac{\eps}{2}\right)\frac{(2\ell+s)^2}{2}-\ln\left(\frac{n-2k}{k}\right)(2\ell+s).
\end{align}

Our aim is to bound $F$ from above ---
we will have two cases depending on how large $2\ell+s$ is.\\\\ 
\textbf{Case I: $2\ell+s=\omega_c\left(\frac{1}{\ln c}\right)k$.}\\
In this case the last term will outweigh all the positive terms in expression~\eqref{eq:fnkls}, so $F$ will be negative. Indeed,
\[ \ln\left(\frac{n-2k}{k}\right)(2\ell+s)=(1-o_c(1))(\ln c) (2\ell+s)\]
and we first claim that this expression dominates the first two positive terms.
To see this, set $g(x):=x\ln (k/x)$ and observe that $g'(x)=\ln(k/x)-1$ and $g''(x) = -1/x <0$.
Therefore $g$ attains its maximum when $\ln(k/x)-1=0$, i.e. when $x=k/e$,
which implies that for all $x$, $g(x)\le g(k/e) = k/e$.
Therefore
\[\ell\ln\left(\frac{k}{\ell}\right)+ s\ln\left(\frac{k}{s}\right)\leq \frac{2k}{e}\leq o_c(\ln c)(2\ell+s).\]
Since $s+2s\ln 2\leq 3k =o_c(\ln c)(2\ell+s)$ is also comparatively small,
the only term which remains is
$\frac{c}{n}\left(1+\frac{\eps}{2}\right)\frac{(2\ell+s)^2}{2}$
and by recalling that $2\ell+s\leq 2k$ we get that this term is
less than
$\frac{kc}{n}\left(1+\frac{\eps}{2}\right)(2\ell+s)
\leq \left(1-\frac{\eps}{2}\right)(\ln c)(2\ell+s)$.
Therefore $F$ is dominated by $-\frac{\eps}{2}(\ln c)(2\ell+s)$,
which means that $F\le 0$,
so in this case we are done.\\\\
\textbf{Case II: $2\ell+s=O_c\left(\frac{1}{\ln c}\right)k$.}\\
Here we split $F$ into two parts,
and separately prove that they are small:
\begin{align*}
    F_1&:=\ell\ln\left(\frac{k}{\ell}\right)
    +\frac{c}{n}\left(1+\frac{\eps}{2}\right)(2\ell^2+2\ell s)
    -\ln\left(\frac{n-2k}{k}\right)2\ell;\\
    F_2&:=s\ln\left(\frac{k}{s}\right)+s +2s\ln 2
    +\frac{cs^2}{2n}\left(1+\frac{\eps}{2}\right)
    -\ln\left(\frac{n-2k}{k}\right)s.
\end{align*}

\smallskip
\textbf{Upper bound on $F_1$:} We will again use the fact that,
by the arguments in Case~I, $g(x)=x\ln(k/x)$
is increasing if $x \le \frac{k}{c\ln c} \le k/e$.
We divide further into two subcases.

\textbf{Subcase (a): $\ell\leq \frac{k}{c\ln c}$.}
Here we simply ignore the negative term for an upper bound, and also observe
that $2\ell^2 + 2\ell s \le 2 (2\ell+s)\ell \le 2k\ell$, which gives
\begin{align*}
    F_1 \le \ell\ln\left(\frac{k}{\ell}\right)
    +\frac{c}{n}\left(1+\frac{\eps}{2}\right)2k\ell
    & \leq \frac{k}{c\ln c}\ln (c\ln c)+\frac{3k^2}{n\ln c} \\
    & \le k\left(\frac{2}{c} +\frac{3k}{n\ln c}\right)\\
    & \le \left(\frac{\ln c}{c}\right)n\cdot \frac{5}{c} \\
    & = o_c\left(\frac{1}{c}\right)n.
\end{align*}

\textbf{Subcase (b): $\ell> \frac{k}{c\ln c}$.}
In this case we bound the positive terms
in $F_1$ by
\begin{align*}
    \ell\ln\left(\frac{k}{\ell}\right)
    +\frac{c}{n}\left(1+\frac{\eps}{2}\right)(2\ell^2+2\ell s)
    & \le \ell\ln(c\ln c)+\frac{c}{n}3\ell(2\ell+s) \\
    & \leq  \frac{4\ln c}{3}\ell
    +\frac{3ck}{n}\ell O_c\left(\frac{1}{\ln c}\right) \\
    & \le \frac{4\ln c}{3}\ell
    +(\ln c)\ell O_c\left(\frac{1}{\ln c}\right) \\
    & \leq \frac{3}{2}\ell\ln c .
\end{align*}
Meanwhile, (the absolute value of) the negative term in $F_1$ is
$$
\ln\left(\frac{n-2k}{k}\right)2\ell = \ln\left(\frac{c}{(1-\eps)\ln c}-2\right)2\ell
\ge \frac{3}{2}(\ln c)\ell,
$$
and so in total we have $F_1\le 0$.

Thus in both subcases (a) and (b), we certainly have
$$
F_1 \le o_c\left(\frac{1}{c}\right)n.
$$

\par \textbf{Upper bound on $F_2$:} As before, we will consider two subcases.

\textbf{Subcase (a): $s\leq \frac{k}{c\ln c}$.}
Analogously as for $F_1$ we get that $F_2\leq o_c\left(\frac{1}{c}\right)n$.
More precisely, again using the fact that $g(x)=x\ln(k/x)$ is increasing for $x \le \frac{k}{c\ln c}$, we have
\begin{align*}
    F_2 & \le s\ln\left(\frac{k}{s}\right)+s +2s\ln 2
    +\frac{cs^2}{2n}\left(1+\frac{\eps}{2}\right) \\
    & \le \frac{k}{c\ln c}\left(\ln(c\ln c) + 1 + 2\ln 2
    + \frac{k}{(\ln c)n} \right)\\
    & \le \frac{n}{c^2}\left(2\ln c + \frac{1}{c} \right)\\
    & = o_c\left(\frac{1}{c}\right)n.
\end{align*}

\textbf{Subcase (b): $s> \frac{k}{c\ln c}$.}
First we will estimate the last term in $F_2$
using the assumption of Case II, which implies that
$s= O_c\left(\frac{1}{\ln c}\right)k$:

\begin{align}\label{eq:H2larges2}
      \ln\left(\frac{n-2k}{k}\right)s &=\ln\left(\frac{c}{(1-\eps)\ln c}
      \left(1-\frac{2(\ln c)(1-\eps)}{c}\right)\right)s \nonumber \\
         &\ge \ln\left(\frac{c}{\ln c}\right)s  +
            \ln\left(1-\frac{2(\ln c)(1-\eps)}{c}\right)s \nonumber \\
            & = \big(\ln c-\ln\ln c\big)s  -
            O_c\left(\frac{\ln c}{c}\right)O_c\left(\frac{1}{\ln c}\right)k\nonumber \\
    &=\big(\ln c-\ln\ln c \big)s-o_c\left(\frac{1}{c}\right)n.
\end{align}
Now set $\alpha:= \frac{cs}{k\ln c}$, so that
$s=\alpha \frac{\ln c}{c}k$ and we have
$\frac{1}{(\ln c)^2} < \alpha \leq O_c\left(\frac{c}{(\ln c)^2}\right)$,
since $\frac{k}{c\ln c} < s =O_c\left(\frac{k}{\ln c}\right)$. 
This implies that
\begin{align}\label{H_2bound}
F_2&= s\ln\left(\frac{c}{\alpha\ln c}\right)+s+2s\ln 2+\frac{cs^2}{2n}\left(1+\frac{\eps}{2}\right)-\big(\ln c-\ln\ln c-\ln (1-\eps)\big)s + o_c\left(\frac{1}{c}\right)n\nonumber\\
&\leq s\ln(1/\alpha)+s+2s\ln 2+\frac{cs^2}{n} + o_c\left(\frac{1}{c}\right)n.
\end{align}
First suppose $\alpha>c^{1/4}$. Using the assumption of Case II we have $\frac{cs^2}{n}=O_c(1)s$, and so we obtain
\[
F_2 \le s\left(-\frac{\ln c}{4}+O_c(1)\right) + o_c\left(\frac{1}{c}\right)n
\le o_c\left(\frac{1}{c}\right)n.
\]

On the other hand, if $\alpha<c^{1/4}$ then we have
$$
\frac{cs^2}{n}=c\frac{(\alpha k(\ln c)/c)^2}{n} \le \frac{(\ln c)^2 k^2}{n\sqrt{c}}
\le \frac{(\ln c)^4 n}{c^{5/2}} = o_c\left(\frac{1}{c}\right)n,
$$
and substituting this into \eqref{H_2bound}, we obtain
\begin{align*}
F_2 & \le s\left(\ln(1/\alpha) + 1 + 2\ln 2 \right)
+ o_c\left(\frac{1}{c}\right)n\\
& = \frac{k\ln c}{c}  \left(\alpha \ln(1/\alpha) + \alpha + 2\alpha\ln 2\right)
+ o_c\left(\frac{1}{c}\right)n,
\end{align*}
and this function is maximized for $\alpha=4$. Thus we have 
\[
    F_2\leq \frac{k\ln c}{c}\cdot 4 + o_c\left(\frac{1}{c}\right)n
    = o_c\left(\frac{1}{c}\right)n.
\]
Thus in all cases we have $F_1,F_2 \le o_c\left(\frac{1}{c}\right)n$,
and therefore also $F= F_1+F_2 = o_c\left(\frac{1}{c}\right)n$.
Substituting this into~\eqref{exponent}, we deduce that
$$
b_{\ell,s} \le \exp \left(o_c\left(\frac{1}{c}\right)n\right)
\le \exp\left(\frac{n}{3c}\right)
$$
as claimed.
\end{proof}{}

We have now collected all the auxiliary results we need, and we show that these
imply the various previously stated results.
First we show that \Cref{prop:boundbklspositive} implies \Cref{lem:boundbkls}.

\begin{proof}[ of \Cref{lem:boundbkls}]
Clearly \Cref{lem:boundbkls} is implied directly by \Cref{prop:boundbklspositive}
for any $\ell,s\ge 1$ such that $\ell+s\le k-1$.
The remaining terms can be dealt with by combining \Cref{prop:boundbklspositive}
and \Cref{prop:boundbratios},
together with the observation that
$$
n^{9}\exp\left(\frac{n}{3c}\right)\leq \exp\left(\frac{n}{2c}\right)
$$
because $c<\frac{n}{(\ln n)^3}$.
\end{proof}{}

We can now combine previous results to prove \Cref{lem:secondmoment}.

\begin{proof}[ of \Cref{lem:secondmoment}]
Applying \Cref{cor:explicitsum} and \Cref{lem:boundbkls},
we have
\begin{align*}
    \frac{\sum_{i\in I}\EE[X_i|X_1=1]}{\EE[Y_k]} = \sum_{\ell=0}^k \sum_{s=0}^{k-\ell} b_{\ell,s}
    \le \sum_{\ell=0}^k \sum_{s=0}^{k-\ell}\exp\left(\frac{n}{2c}\right)
    \le n^2 \exp\left(\frac{n}{2c}\right)
    \le \exp\left(\frac{n}{c}\right)
\end{align*}
as claimed.
\end{proof}

\section{Concluding remarks}

In this paper we asymptotically determined the size of the largest induced matching of $G(n,p)$. Using a similar approach, one could probably show the existence of large forests with components of bounded size, but the calculations get messy very quickly. When instead of a large induced matching we look for a fixed large induced bounded degree tree, much less is known. Nevertheless, based on the evidence listed in the introduction, we propose the following conjecture:

\begin{conj}\label{conj:trees}
Let $\Delta\geq 2$, let $\eps>0$ and let $p=p(n)$ be a function such that $\frac{C}{n}<p<0.99$ where $C=C(\Delta,\eps)$ is sufficiently large. Let $T$ be a tree with 
$(2-\eps)\log_q(np)$ vertices and maximum degree $\Delta$, where $q=1/(1-p)$. Then with high probability $G(n,p)$ contains $T$ as an induced subgraph.
\end{conj}

If true, this conjecture would be asymptotically best possible.
As mentioned in the introduction, the result was already proved for $p\geq n^{-1/2}(\ln n)^2$ in \cite{draganic2020large}.

\bibliographystyle{siam}
\bibliography{ref}

\newpage

\begin{appendices}

\section{Proof of \Cref{prop:boundbratios}}\label{app:ratios}

In order to prove the claim,
we will show that
for any $0\le \ell,s \le k-1$ we have:
\begin{align*}
    b_{\ell,0} & \le n\cdot b_{\ell,1};\\
    b_{0,s} & \le n^3\cdot b_{1,s};\\
    b_{\ell,k-\ell} & \le n^2\cdot b_{\ell,k-\ell-1},
\end{align*}
and furthermore:
\begin{align*}
    b_{k,0} & \le n^2\cdot b_{k-1,0};\\
    b_{k-1,0} & \le n^2\cdot b_{k-2,0}.
\end{align*}
This suffices since for any $\ell,s$,
by concatenating at most three such inequalities
we may relate $b_{\ell,s}$ to some
$b_{i,j}$ where $1 \le i,j \le k-1$ and $i+j \le k-1$.
For example, we have 
$$
b_{k-1,1} \le n^2 \cdot b_{k-1,0}
\le n^2 \cdot n^2 \cdot b_{k-2,0}
\le n^2 \cdot n^2 \cdot n \cdot b_{k-2,1}.
$$
(In fact, this is the most convoluted concatenation necessary -- all others are far more natural.)
The ``cost'' associated with each inequality is at most $n^3$,
and therefore the cost of concatenating three is at most
$(n^3)^3 = n^9$, as claimed.

Now to prove the inequalities, observe from~\eqref{eq:bkls} that
\begin{align*}
\frac{b_{\ell,0}}{b_{\ell,1}} & = \frac{1}{4}\frac{((k-\ell-1)!)^2}{((k-\ell)!)^2}
\frac{(n-4k+2\ell+1)!}{(n-4k+2\ell)!} (1-p)^{\binom{2\ell+1}{2}-\binom{2\ell}{2}}\\
& = \frac{1}{4} \frac{n-4k+2\ell+1}{(k-\ell)^2}(1-p)^{2\ell}\\
& \le n.
\end{align*}
Similarly,
\begin{align*}
    \frac{b_{0,s}}{b_{1,s}} & = \frac{1}{2} \frac{1}{(k-s)^2} (n-4k+s+2)(n-4k+s+1) \frac{(1-\eps)\ln c}{k} (1-p)^{2s}\\
    & \le n^2 \ln c \le n^3.
\end{align*}
To prove the remaining inequalities we will use the bound
$$(1-p)^{-k} \le \exp(pk) = \exp(\ln c) \le n.$$
Thus we obtain
\begin{align*}
    \frac{b_{\ell,k-\ell}}{b_{\ell,k-\ell-1}} & = 4 \frac{1}{(k-\ell)} \frac{1}{(n-3k+\ell)}
    (1-p)^{-(k+\ell-1)}
    \le \frac{5}{n} (1-p)^{-2k}
    \le n^2.
\end{align*}
For the final two inequalities, observe that
\begin{align*}
    \frac{b_{k,0}}{b_{k-1,0}} & = 2\frac{1}{k}\frac{1}{(n-2k)(n-2k-1)}\frac{k}{(1-\eps)\ln c}(1-p)^{4-4k}\\
    & \le \frac{3}{n^2 \ln c} \cdot n^4 \le n^2,
\end{align*}
while
\begin{align*}
    \frac{b_{k-1,0}}{b_{k-2,0}}
    & = 2 \frac{4}{k-1}
        \frac{1}{(n-2k-2)(n-2k-3)}
        \frac{k}{(1-\eps) \ln c}
        (1-p)^{8-4k} \\
    & \le \frac{9}{n^2 \ln c} \cdot n^4 \le n^2,
\end{align*}
as required.\qed

\section{Proof of of \Cref{main} for large $p$}\label{app:secondmoment}

In this appendix we prove \Cref{main} in the
case when $\frac{1}{(\ln n)^3} \le p = o(1)$.
As before we let
$\eps_0>0$, let $\eps := \frac{\eps_0}{3}$ and let
$p=c/n$ where now $c=c(n)$ satisfies
$\frac{n}{(\ln n)^3} \le c =o(n)$. 
As before, let
\begin{equation*}
k:=\frac{(1-\eps)\ln c}{c}n.
\end{equation*}
Note that we now have $\ln n \ll k \le (\ln n)^4$.

We first observe that some calculations from the proof for smaller $p$ are still valid. In particular, from the proof of \Cref{probability2}
we have the expression
$$
\EE[Y_k^2] = \EE[Y_k] \sum_{i \in I} \EE[X_i | X_1=1],
$$
and therefore by \Cref{cor:explicitsum} we have
\begin{equation}\label{eq:momentratio}
\frac{\EE[Y_k^2]}{\EE[Y_k]^2} = \sum_{\ell=0}^k \sum_{s=0}^{k-\ell}b_{\ell,s},
\end{equation}
where as before we have
\begin{equation}\label{eq:blsdef2}
    b_{\ell,s}  :=  2^{\ell+2s}\frac{(k!)^2}{\ell!s!((k-\ell-s)!)^2}\cdot\frac{((n-2k)!)^2}{n!(n-4k+2\ell+s)!}\left(\frac{k}{(1-\eps)\ln c}\right)^{\ell}(1-p)^{-\binom{2\ell+s}{2}+\ell}.
\end{equation}
Observe that $2^\ell/\ell! \le 2$ and $4^s/s! \le 64/6 \le 11$ for all $\ell,s \in \mathbb{N}_0$, and so we have
\begin{equation}\label{eq:powerandfactorial}
    \frac{2^{\ell+2s}}{\ell! s!} \le 22.
\end{equation}
Furthermore, we obtain
\begin{equation}\label{eq:klsfactorials}
\frac{(k!)^2}{((k-\ell-s)!)^2} \le k^{2(\ell+s)}
\end{equation}
and 
\begin{align}\label{eq:nfactorials}
\frac{((n-2k)!)^2}{n!(n-4k+2\ell+s)!} & \le \frac{1}{\left(1+\frac{2k}{n-2k}\right)^{2k}\left(1-\frac{2k}{n-2k}\right)^{2k} (n-4k)^{2\ell+s}} \nonumber\\
& = \left(1-\frac{(2k)^2}{(n-2k)^2}\right)^{-2k} \frac{\left(1-\frac{4k}{n}\right)^{-(2\ell+s)}}{n^{2\ell+s}} \nonumber \\
&\le \exp \left(\frac{(2k)^3}{(n-2k)^2} + \frac{4k(2\ell+s)}{n}\right) \frac{1}{n^{2\ell+s}} \nonumber\\
& = (1+o(1)) \frac{1}{n^{2\ell+s}}.
\end{align}

Finally, we have 
\begin{equation}\label{eq:n/c}
\frac{k}{(1-\eps)\ln c} = \frac{n}{c}
\end{equation}
and
\begin{equation}\label{eq:1-p}
    (1-p)^{-\binom{2\ell+s}{2}+\ell} \le \exp \left( p\frac{(2\ell+s)^2)}{2} \right) = \exp\left(\frac{c}{n}\left(2\ell(\ell+s) +\frac{s^2}{2}\right)\right).
\end{equation}
Substituting Equations~\eqref{eq:powerandfactorial}--\eqref{eq:1-p} into~\eqref{eq:blsdef2}, we obtain

\begin{align*}
b_{\ell,s}
& \le 22 k^{2(\ell+s)} (1+o(1)) \frac{1}{n^{2\ell+s}}
\left(\frac{n}{c}\right)^\ell
\exp\left(\frac{c}{n}\left(2\ell(\ell+s) +\frac{s^2}{2}\right)\right) \\
& \le 23\left(\frac{k^2}{nc}\exp\left(\frac{2c(\ell+s)}{n}\right)\right)^\ell
\left(\frac{k^2}{n}\exp\left(\frac{cs}{2n}\right)\right)^s \\
& \le 23\left(\frac{(\ln n)^8}{nc}\exp\left(\frac{2ck}{n}\right)\right)^\ell
\left(\frac{(\ln n)^8}{n}\exp\left(\frac{ck}{2n}\right)\right)^s \\
& = 23\left(\frac{(\ln n)^8}{nc}c^{2(1-\eps)}\right)^\ell
\left(\frac{(\ln n)^8}{n}c^{\frac{1}{2}(1-\eps)}\right)^s \\
& \le 23 \left(\frac{(\ln n)^8}{n^{2\eps}}\right)^\ell
\left(\frac{(\ln n)^8}{\sqrt{n}}\right)^s \\
& \le n^{-\eps(\ell+s)} \qquad \mbox{if } \ell+s >0.
\end{align*}

For $\ell=s=0$, we need to be a little more careful --- here we can directly observe that
$$
b_{0,0} = 1 \cdot \frac{(k!)^2}{1\cdot (k!)^2}
\cdot \frac{((n-2k)!)^2}{n!(n-4k)!} \cdot 1 \cdot 1 \stackrel{\eqref{eq:nfactorials}}{=} 1+o(1).
$$
Thus~\eqref{eq:momentratio} becomes
\begin{align*}
    \frac{\EE[Y_k^2]}{\EE[Y_k]^2} \le (1+o(1)) \sum_{\ell=0}^k \sum_{s=0}^{k-\ell}n^{-\eps(\ell+s)}
    & = (1+o(1)) \left(\sum_{\ell=0}^k n^{-\eps\ell}\right) \left(\sum_{s=0}^{k-\ell} n^{-\eps s}\right)
    = (1+o(1)).
\end{align*}
Therefore by Chebyshev's inequality we have
$$
\Pr(Y_k = 0) \le \frac{\mathrm{var[Y_k]}}{\EE[Y_k]^2}
= \frac{\EE[Y_k^2]-\EE[Y_k]^2}{\EE[Y_k]^2}
= o(1),
$$
as required.

\end{appendices}

\end{document}